\documentclass{report}

\usepackage{amssymb,amsmath,amsthm}
\usepackage{graphicx}
\usepackage{amsfonts}
\usepackage{stmaryrd}

\author{Maciej Satkiewicz}
\title{Transfinite Asymptotic Dimension}
\date{}

\newcommand{\m}[1]{\mathbb{#1}}
\newcommand{\mc}[1]{\mathcal{#1}}

\newcommand{\wtw}{\Leftrightarrow}

\newtheorem{twr}{Theorem}
\newtheorem{lem}{Lemma}
\newtheorem{stw}{Assertion}
\newtheorem{df}{Definition}
\newtheorem{spos}{Observation}
\newtheorem{wn}{Corollary}
\newtheorem{uw}{Note}
\newtheorem{fkt}{Fact}
\newtheorem{hip}{Hipothesis}

\begin{document}
\maketitle
\tableofcontents 

\chapter*{Introduction}
\addcontentsline{toc}{chapter}{Introduction}

Asymptotic property C for metric spaces was introduced by Dranishnikos as generalization of finite asymptotic dimension - asdim. It turns out that this property can be viewed as transfinite extension of asymptotic dimension. The original definition was given by Radul (see~\cite{Rad}). We introduce three equivalent definitions, show that asymptotic property C is closed under products (open problem stated in~\cite{Pearl}) and prove some other facts. Some examples of spaces enjoying countable trasfinite asymptotic dimension are given. We also formulate open problems and state ''omega conjecture'', which inspired most part of this paper. 

\chapter{Three equivalent definitions}
\section{Original definition}\label{sec:pierdef}

Our basic assumption is that metric is finite.
\\\\
The family $\mc{A}$ of subsets of a metric space is called \emph{uniformly bounded} if there exists $C\geq 0$ such that diam$A\leq C$ for every $A\in\mc{A}$; $\mc{A}$ is called \emph{r-disjoint} for some $r>0$ if $d(A_{1},A_{2})\geq r$ for every $A_{1},A_{2}\in\mc{A}$ such that $A_{1}\neq A_{2}$.
\emph{Asymptotic dimension} of metric space $X$ doesn't exceede $n\in\m{N}\cup\{0\}$ (we write asdim$X\leq n$) $\wtw$ for every $r>0$ exists finite set $\mc{U}$ consisting of uniformly bounded families $U_{1},...,U_{n}$\footnote{We shall call $\mc{U}$ briefly an \emph{uniformly bounded family} or \emph{uniformly bounded cover} if $X\subset\bigcup\mc{U}$; our terminology shouldn't raise any confusion later on; note, that we consider only finite $\mc{U}$.}, which sum covers $X$ (so $\mc{U}$ is uniforlmy bounded cover of $X$) and every $U_{i}$ is r-disjoint.

Now we definie transdinite asymptotic dimension (denoted by trasdim).

Let $L$ be an arbitrary set. Let $FinL$ denote the set of all finite nonempty subsets of $L$; $M\subset FinL$. For $\sigma\in\{\emptyset\}\cup FinL$ we definie:

$$M^{\sigma}=\{\tau\in FinL~|~\tau\cup\sigma\in M~\textrm{and}~\tau\cap\sigma=\emptyset\}$$

We write $M^{a}$ instead of $M^{\{a\}}$ for $a\in L$.
Define inductively ordinal number Ord$M$ as follows:

$$\textrm{Ord}M=0\wtw M=\emptyset$$
$$\textrm{Ord}M\leq\alpha\wtw\forall_{a\in L} \textrm{Ord}M^{a}<\alpha$$
$$\textrm{Ord}M=\alpha\wtw\textrm{Ord}M\leq\alpha~\textrm{it is not true that Ord}M<\alpha$$
$$\textrm{Ord}M=\infty\wtw\textrm{Ord}M>\alpha~\textrm{for every ordinal $\alpha$}$$ 

For given metric space $(X,d)$ define set:

$$A(X,d)=\{\sigma\in\textrm{Fin}\m{N}~|~\textrm{there is no uniformly bounded cover $\mc{U}$ of $X$,}$$
$$\textrm{consisting of families $\mc{V}_{i}$ for $i\in\sigma$ such that each $\mc{V}_{i}$ is i-disjoint}\}$$

Now we define trasdim$X$=Ord$A(X,d)$. It is easy to see, that trasdim is indeed generalization on asdim. We'll see this clearly later on.

Dranishnikov introduces asymptotic property $C$ as follows: metric space $X$ has asymptotic property $C$ if for every (strictly increasing) infinite sequence of natural numbers $n_{1}<n_{2}<...$ exists uniformly bounded cover of $X$ consisted of families $\mc{V}_{i}$ for $1\leq i\leq n$ where $n\in\m{N}$ and each $\mc{V}_{i}$ is i-disjoint. It turns out that $X$ has asymptotic property $C \wtw$ exists ordinal number $\alpha$ such that trasdim$X\leq\alpha$ (so trasdim$X<\infty$). We'll see this soon, too.

\section{Definition by rank of a tree}

Let $\omega*$ denote the set of all finite sequences of natural numbers (including empty sequence - denoted by $\emptyset$). By $dom(s)$ we denote domain of sequence $s$ and i-th element of $s$ is as usual $s(i)$; numbering starts from 1. \emph{Tree} (we'll consider only trees on $\omega$) is a set $T\subset\omega*$ closed under prefixes: $\forall_{s\in T}~t\subset s\Rightarrow t\in T$. In particular, every nonempty tree includes empty sequence. Partial ordering of a tree is given by reverse inclusion. Every minimal element of this relation is called \emph{leaf}. $s^{\wedge} n$ denotes sequence obtained from extending $s$ by one element $n$ (so that $s^{\wedge} n$ is \emph{son} of $s$).

Define inductively an ordinal number $rank_T(s)$ (rank) for $s\in T$:
$$\textrm{rank}_{T}(s)=0\wtw\textrm{s is a leaf}$$
$$\textrm{rank}_{T}(s):=sup\{\textrm{rank}(s^{\wedge} n)+1~|~n\in\m{N}\}$$

The rank of the whole tree is defined as: rank$T:=\textrm{rank}_{T}(\emptyset)$. Rank of a tree $T$ is well defined $\wtw$ $T$ is well founded, that means $T$ doesn't have infinite branch. Otherwise we put rank$T:=\infty$. Rank of a tree defined on $\omega$ can have only countable values. We prove both of these facts:

\begin{stw}\label{stw:ranga}If tree $T$ is well founded, then rank$T$ is defined and equal to a certain countable ordinal number.\end{stw}

To prove this we'll need some other facts. Let $<_{lex}$ be lexicographic order on $\omega*$. Define

\begin{df}Kleene-Brouwer ordering ($<_{kb}$) is an ordering on $\omega*$ fulfilling the condition:\end{df}
$$s<_{kb}t\wtw s<_{lex}t~\textrm{or}~(s=_{lex}t~\textrm{i}~t\subsetneq s)$$  

\begin{lem} If tree $T$ is well founded then relation $(<_{kb})_{|T}$ is well ordering (note that the opposite implication is obvious).\end{lem}

\begin{proof}
By contraposition. Suppose that $<_{kb}$ restricted to $T$ is not well ordering. Then there exists strictly decreasing infinite sequence. It's easy to see, that every such sequence in a tree (recall that we consider only trees on $\omega*$) has subsequence consisting of elements of one branch of the tree. Therefore $T$ has an infinite branch and is not well founded. 
\end{proof}

Before we prove assertion~\ref{stw:ranga}, we define rank of a tree in another way, proving the equivalence of both definitions and the assertion at once.

Let $T$ - well founded tree. $T_{0}$ denotes the set of leaves of $T$. For $\alpha>0$ define:

$$T_{\alpha}=\textrm{set of leaves of}~T\backslash\bigcup_{\beta<\alpha}T_{\beta}$$

Rank ($Rank_{T}(s)$) of element $s\in T$ is the unique ordinal $\alpha$, such that $s\in T_{\alpha}$. Rank of the whole tree is by definition equal to $Rank_{T}(\emptyset$). $T$ is well founded and countable and therefore it's clear that $Rank_{T}(s)$ is well defined for every $s\in T$. If $T$ wasn't well founded, we would put Rank$T=\infty$.

\begin{stw}\label{stw:ranga2} For every well founded tree $T$ and every element $s\in T$ $rank_{T}(s)=Rank_{T}(s)$.\end{stw}

\begin{proof}
(Of assertions~\ref{stw:ranga}~and~\ref{stw:ranga2}) By induction on $(<_{kb})_{|T}$. The least element of the tree is a leaf, so its rank is defined; clearly both definitions coincide. Suppose $s$ is $\alpha$-th element of $(<_{kb})_{|T}$ and let our assertion be true for elements lesser than $s$. \emph{A fortiori} its true for each son of $s$. Obviously it allows us to define $rank_{T}(s)$. Note that $s\in T_{sup\{rank(s^{\wedge} n)+1~|~n\in\m{N}\}}$ and so $s\in T_{rank_{T}(s)}$, which concludes our thesis. $T$ is countable and so, by second definition, rank cannot have uncountable values.
\end{proof}

For given metric space $(X,d)$ define three trees:

$$T(X,d)=\{s\in\omega*~|~\textrm{there is no uniformly bounded cover $\mc{U}$ of $X$ consisting of}$$
$$\textrm{families $\mc{V}_{i}$ for $i\in dom(s)$ such that every $\mc{V}_{i}$ is s(i)-disjoint}\}$$
$$T(X,d)^{\rightarrow}=\{s\in\omega*~|~\textrm{there is no uniformly bounded cover $\mc{U}$ of $X$ consisting of}$$
$$\textrm{families $\mc{V}_{i}$ for $i\in dom(s)$ such that every $\mc{V}_{i}$ is s(i)-disjoint and $s$ is nondecreasing}\}$$
$$T(X,d)^{\nearrow}=\{s\in\omega*~|~\textrm{there is no uniformly bounded cover $\mc{U}$ of $X$ consisting of}$$
$$\textrm{families $\mc{V}_{i}$ for $i\in dom(s)$ such that every $\mc{V}_{i}$ is s(i)-disjoint and $s$ is strictly increasing}\}$$

\begin{df}Cover corresponding to finite sequence $s$ will be called s-cover. Similarly a family corresponding to $s$ - not necessarily a cover - will be called s-family. Note, that s-family is a \emph{sequence} of sufficently disjoint families of subsets of given metric space.\end{df}

It's easy to see that ranks of all tree threes above are equal. It follows from the fact, that for every sequence $s\in T(X,d)$ there exists strictly increasing sequence $S\in T(X,d)^\nearrow$ scuh that $dom(s)\subset dom(S)$ and $s(i)\leq S(i)$ for $i\in dom(s)$. We shall prove this in more general version.

\begin{df}t-embbeding is a function f whose domain and codomain are trees; such that two conditions are satisfied: $|dom(s)|=|dom(f(s))|$ (where $|\cdotp|$ denotes power of set) and $t\subset s\Rightarrow f(t)\subset f(s)$.\end{df}

\begin{spos}The image of a tree by t-embbeding is a tree.\end{spos}

\begin{proof}
Note that empty sequence must be mapped onto empty sequence. Now the closure under prefixes of the image of f is clear. 
\end{proof}

\begin{twr}\label{twr:twl}For trees $T$ i $Y$ (not necessarily well founded) if there exists t-embbeding $f:T\rightarrow Y$ then $rankT\leq rankY$.\end{twr}

\begin{proof}
It is sufficient to show theorem for $L:=f[T]\subset Y$, because rank is monotonic in respect to inclusion of trees. If $T$ has infinite branch, then $L$ has infinite branch too. Let's assume then, that $T$ and $L$ are well founded (if $L$ is not, then our thesis is automatically true). We show that $\forall_{s\in T}~rank_T(s)\leq rank_L(f(s))$ - by induction on $rank_T(s)$. Zeroth step is obvious. Every other step is also easy - inductive assumption is in force for sons of $s$, which are mapped on sons of $f(s)$.
\end{proof}

In our case let's take the function (going from any of the three defined earlier trees to $T(X,d)^{\nearrow}$) which maps every sequence $s$ into lexicographically least strictly increasing sequence $S$, having the same lenght as ours and such that $s(i)\leq S(i)$ for $i\in dom(s)$. It is t-embbeding, because such sequence $S$ is always member of $T(X,d)^{\nearrow}$.

We put trasdim$X=$rank$T(X,d)$. Let's convince ourselves that both definitions of trasdim are equivalent.

\begin{spos}\label{spos:lacznosc} Let $\sigma,\tau\in FinL$ and $\sigma\cap\tau=\emptyset$. Then $M^{\sigma\cup\tau}=(M^{\sigma})^{\tau}$. If $\sigma\cap\tau\neq\emptyset$ then $(M^{\sigma})^{\tau}=\emptyset$\end{spos}

By natural identification of set with strictly increasing sequence we can for a space $(X,d)$ construct a tree basing on set $A(X,d)$ in the following way (the father($\sigma$) is the father of that sequence in any tree such that $\sigma$ belongs to that tree):
$$\sigma\in TA(X,d)\wtw M^{father(\sigma)}\neq\emptyset$$ 

\begin{twr}If $TA(X,d)$ is well founded then $T(X,d)^\nearrow$ is also well founded and we have $OrdA(X,d)^{\sigma}=rank_{TA(X,d)}(\sigma)=rank_{T(X,d)^{\nearrow}}(\sigma)$ for $\sigma\in TA(X,d)$. Also OrdA(X,d)$=\infty\wtw rankTA(X,d)=\infty\wtw rankT(X,d)=\infty$ are satisfied.\end{twr}

\begin{proof}
We identify sets with strictly increasing sequences. Firstly lets prove the equality\\ $OrdA(X,d)^{\sigma}=rank_{TA(X,d)}(\sigma)$. Notice that $OrdA(X,d)^{\sigma}=0\wtw(\sigma\notin TA(X,d)$ or $\sigma$ is a leaf of $TA(X,d)$)$\wtw(\sigma\notin TA(X,d)$ or rank$_{TA(X,d)}(\sigma)=0$). Also note that OrdM=sup$\{Ord(M^n)+1~|~n\in\m{N}\}$. Using observation~\ref{spos:lacznosc} we have for $\sigma\in TA(X,d)$ equalities
$OrdA(X,d)^{\sigma}=sup\{Ord(A(X,d)^{\sigma})^n+1~|~n\in\m{N}\}=sup\{Ord(A(X,d)^{\sigma^{\wedge n}})+1~|~n\in\m{N}\}$. It suffices to replace the character string ''$OrdA(X,d)$'' by string ''$rank_{TA(X,d)}$'' to realise that our equality holds for every $\sigma\in TA(X,d)$. The second equality and first implication of our thesis stem from the fact that trees $TA(X,d)$ and $T(X,d)^{\nearrow}$ are identical (excluding some neglectable set of leaves of $TA(X,d)$ which don't belong to $T(X,d)^{\nearrow}$ and for which our equality has no sense). Now we have also proven equivalences from our statement, because Ord$M=$Ord$M^\emptyset=rank_{TA(X,d)}(\emptyset)=$rankTA(X,d).
\end{proof}

Preceding theorem proves that both definitions of trasdim are equivalent.

\section{Definition by winning strategy in dimension game}\label{sec:gra}

There are two players, player $A$ and player $B$. A metric space $X$ is given. In the first round of ''dimension game'' for space $X$ player $B$ chooses a natural number $r_{1}$ and player $A$ responds with uniformly bounded cover of $X$ which is $r_{1}$-disjoint and gives number $k_{1}$ - the number of families in that cover.

In general, in n-th round player $B$ chooses number $r_{n}\geq r_{n-1}$ and player $A$ has to repsond with uniformly bounded cover $\mc{U}$ of $X$ consisting of families $V_{1},V_{2},...,V_{k_{n}}$ such that $V_{i}$ is $r_{i}$-disjoint for $1\leq i\leq n$, $V_{j}$ is $r_n$-disjoint for $n\leq j\leq k_{n}$ and give number $k_{n}$ - the power of that cover.

The game ends in n-th round if $k_{n}=n$ and then player $A$ wins. Otherwise - if game doesn't end in any round or player $A$ cannot respond with proper cover - player $B$ wins.

\begin{spos}Trasdim$X<\infty\wtw$ player A has a winning strategy.\end{spos}

Winning strategy for player $A$ can be depicted as a tree. Take any nondecreasing infinite sequence $c\in\omega^{\omega}$. Suppose that this sequence is a strategy of player $B$. There exists least $n_{c}$ such that the game can end in $n_{c}$-th round. Let $g(c)$ be the set consisting of all proper prefixes of sequence $(c(1),...,c(n_{c}))$. We define

$$G(X,d)=\bigcup_{c\in\omega^{\omega};~c~nondecreasing}g(c)$$

Easily (from definition of these trees) we have $G(X,d)=T^{\rightarrow}(X,d)$. We put \\trasdim$X:=$rank$G(X,d)$.

\section{Some observations}

With our terminology the reader can easily check that trasdim is indeed generalization of asdim; countable value of trasdim is equivalent to asymptotic property C; it's monotonic in respect to inclusion (trasdim of subspace cannot be greater that trasdim of the whole space). 
It's usefull to realize the characteristic, ''monotonne'' structure of dimension tree $T(X,d)$:

\begin{df}The root of the subtree $L$ (rootL) is the greatest (with repsect to inclusion) element belonging to the intersection of all branches (maximal chains) of the subtree.\end{df} 

\begin{df}Complete subtree is the maximal subtree (with repsect to inclusion) with given root.\end{df}

Let $zw(L)$ dentote the set of partial functions: 
$$\gamma\in zw(L)\wtw\exists_{\tau\in L}~\tau\backslash rootL=\gamma$$

\begin{df}The matrix of subtree $L$ is a tree $wL$ obtained from $zw(L)$ by translating the domain: \end{df}

$$wL=\{\gamma(i+|dom(rootL)|)\in\omega*~|~\gamma\in zw(L)\}$$

\begin{twr}\label{twr:fraktal} For space $(X,d)$ and tree $T(X,d)$ the matrices of complete subtrees, whose roots have the same power and differ only on one coordinate, create ascending sequence of sets indexed by this coordinate.\end{twr}

\begin{proof}
It stems from the fact that R-disjoint family is always r-disjoint for $R\geq r$. For example, if there is no r-disjoint cover then \emph{a fortiori} there is no R-disjoint cover. Now our thesis is clear.
\end{proof}

\begin{df}\emph{Moment of stabilization} denotes the round in dimension game when there exists natural number k such that player A can respond with cover of power k no matter what player B does in this round.\end{df}

\begin{uw}By analysis of dimension tree one can easily see that moment of stabilization heralds the end of dimension game, since player B will not be able to force player A to raise the power of covers he can respond with (the opposite assumption leads to contradiction with theorem~\ref{twr:fraktal}).\end{uw}

\chapter{Various results}
\section{Permamence of trasdim}

$X$ and $Y$ are metric spaces. Function $f:X\rightarrow Y$ is called \emph{coarse embbeding} if there are non decreasing, unbounded functions $p_{1},p_{2}:\m{R}_+\rightarrow\m{R}_+$ such that for every pair $x,x'\in X$ we have:
$$p_1(d_X(x,x'))\leq d_Y(f(x),f(x'))\leq p_2(d_X(x,x'))$$
If additionally there exists N such that $Y\subset D_N(f[X])$ (where $D_N$ is a generalised ball of radius N) then $f$ is called \emph{coarse equivalence} and spaces$ X$ and $Y$ are said to be \emph{coarse equivalent}. Coarse equivalence is an equivalence relation.

\begin{twr}If $X$ and $Y$ are coarse equivalent, then $trasdimX=trasdimY$.\end{twr}

\begin{proof}
It suffices to show that trasdim$Y\leq$~trasdimX. Pick a coarse embbeding $f:X\rightarrow Y$ such that $Y\subset B_N(f[X])$ and consider:

\begin{lem}\label{lem:major}If uniformly bounded (say, by number E) families $V_1,V_2,...,V_n$ are repsectively $r_1,r_2,...,r_n$-disjoint and their sum covers $X$ then the families $D_N(f[V_i])$ are repsectively $(p_1(r_i)-2N)$-disjoint, $(p_2(E)+2N)$-bounded and their sum covers $Y$.\end{lem}

The proof is straightforward. Notice that the minium of $((p_1(r_i)-2N))_{i=1}^n$ can be arbitralily large, because $p_1$ goes to infinity. By contraposition we have:

\begin{wn}If sequence $((p_1(r_i)-2N))_{i=1}^m$ belongs to $T(Y,d_Y)$ then the sequence $(r_i)_{i=1}^m$ belongs to $T(X,d_X).$\end{wn}

For every $\sigma\in T(Y,d_Y)$ we can find sequence $(r_i)_{i=1}^{|dom(\sigma)|}$ of lenght $|dom(\sigma)|$ such that \\$(p_1(r_i)-2N)\geq\sigma(i)$ and $((p_1(r_i)-2N))_{i=1}^{|dom(\sigma)|}\in T(Y,d_Y)$ - the inclusion follows form theorem~\ref{twr:fraktal}. Hence, if $T(Y,d_Y)$ has infinite branch then $T(X,d_X)$ has infinite branch, too (proof by induction). Let's assume that $T(Y,d_Y)$ is well founded. Then the ordering $<_{kb}$ on this tree is well ordering and by induction on this ordering we can define t-embbeding of $T(Y,d_Y)$ into $T(X,d_X)$ in the way presented in the beginnig of this paragraph. Now we use theorem~\ref{twr:twl} to conclude the proof.
\end{proof}

\begin{twr}Consider two dimension games for spaces X and Y, played simultaneously, $G_X$ and $G_Y$ repsectively. We don't consider rounds in which game $G_Y$ is finished. If there exists function which, depending on present round and earlier choices of player $B_Y$, maps every possible choice of player $B_Y$ into a choice of player $B_X$ in such a way that player $B_X$ cannot lose unless player $B_Y$ loses - independently of the strategy picked by $B_Y$ - then trasdimY$\leq$trasdimX.\end{twr}

\begin{proof}
Indeed, consider trees $TG(X,d_X)$ and $TG(Y,d_Y)$. Our conditions allow us to define t-embbeding in similar way as in the proof of the former theorem.
\end{proof}

Notice that coarse equivalence is a special case in which the function mentioned in the theorem above does not depend on the round and earlier choices of player $B_Y$ whatsoever.

We conclude this section with simple theorem.

\begin{twr}\label{twr:sksum}Finite sum of spaces having $trasdim<\infty$ also has $trasdim<\infty$.\end{twr}

\begin{proof}
Indeed, finite quantity of dimension games played one after another, in every one of which player A has winning strategy, can be considered as dimension game and player A has a winnig strategy in it.\end{proof}

One can think under what conditions this holds for infinite number of spaces. We provide one possible answer in section~\ref{sec:sumy}.

\section{Dimension of a metric family}\label{sec:wymrodziny}

We say, that family of metric spaces $\Gamma$ (in short - metric family) has dimension Trasdim$\Gamma=\alpha$ if the tree which is set-theoretical sum of dimension trees of spaces belonging to $\Gamma$, that is $L(\Gamma)=\bigcup_{X\in\Gamma}T^{\rightarrow}(X,d_X)$, has countable rank rankL$<\infty$ and the supertree (a superset that is also a tree) of $L(\Gamma)$ consisting of nondecreasing finite sequences $\tau$, for which there is no uniformly bounded familiy of $\tau$-covers for all spaces in $\Gamma$\footnote{that is such family that there is D$\in\m{N}$ - uniform bound for uniformly bounded covers of spaces belonging to $\Gamma$.}, denoted as $T^{\rightarrow}(\Gamma)$, has rank rank$T^{\rightarrow}(\Gamma)=\alpha$\footnote{The reader can easily check that this set indeed is a tree - it follows from the fact that the empty family is r-disjoint and d-bunded for every pair of numbers r and d.}. If the tree $T^\rightarrow(\Gamma)$ is not well founded, then we define Trasdim$\Gamma:=\infty$.

\begin{uw}\label{uw:ogrGamma}Notice that every space belonging to $\Gamma$ can have same dimension tree (for sure such tree t-embbeds into $L(\Gamma)$, so we can assume that all of them are equal to $L(\Gamma)$) but still, with fixed sequence $\tau$, there may be problem with uniform bound for $\tau$-covers of these spaces. As it will turn out, there exist families for which rank~$L(\Gamma)=0$ but rank~$T^{\rightarrow}(\Gamma)$ can have any fixed value.\end{uw}

\emph{Asymptotic sum} as~$\bigcup_{i=1}^{\infty}X_i$ of sequence of metric spaces $(X_i)_{i=1}^{\infty}$ denotes a metric space, which is disjoint sum of $X_i$ such that there exists strictly increasing sequence of natural numbers $(x_i)_{i=1}^{\infty}$ so that for $i<j$ we have $d(X_i,X_j)=x_i+x_{i+1}+...+x_{j-1}+x_j$.

\begin{stw}Asymptotic sum exists for every sequence $(X_i)_{i=1}^{\infty}$ of metric spaces.\end{stw} 

\begin{proof}
It suffices to pick a base point in each space. For distance between points from different spaces we put the sum of their distances from respective base points + the distance between basepoints, which is given by $(x_i)_{i=1}^{\infty}$. It's easy to check the triangle inequality for various cases. In fact it suffices that basepoints create a metric space - in our case they create subspace of real line.
\end{proof}

\begin{df}Asymptotic sum constructed in the proof of previous theorem shall be called \emph{standard asymptotic sum} and denoted by capital letter (As~$\bigcup\Gamma$).\end{df}

\begin{stw}\label{stw:assum}If $\Gamma$ - metric family, Trasdim$\Gamma=\alpha$ and $\Gamma$-countable, then trasdim(as~$\bigcup\Gamma)<\infty$ and trasdim(As~$\bigcup\Gamma)=\alpha$. Moreover, if every space in $\Gamma$ is bounded, then  also trasdim(as~$\bigcup\Gamma)=\alpha$.\end{stw}

\begin{proof}
Pick minimal with respect to inclusion $\tau$ such that there exist uniform bound on uniformly bounded $\tau$-covers of elements of $\Gamma$. Choose $i\in\m{N}$ such that $x_i>max\{\tau(i)~|~i\in dom(\tau)\}$. With little effort we can now cover $X_j$ for $j\geq i$. The remaining part of the space is a finite sum of metric spaces and we can apply theorem~\ref{twr:sksum}. As to equalities from our theorem, in case of bounded spaces we can cover our finite sum with one set, which can be included into some family in our cover of the rest of the space. More general case with unbounded spaces and As~$\bigcup\Gamma$ can be easily dealt with by considering the finite sum of balls $B(b_k,x_i)$ of radious $x_i$ each centered in base points $x_k$ of $X_k$ respectively for $1\leq k\leq i$. Details are left to the reader.
\end{proof}

We'll find the notion of asymptotic sum usefull in constructing examples of spaces of given dimension. As for now however, we prove important theorem.

\begin{twr}\label{twr:wlokna}\emph{Weak Fibering theorem}. Let $f:X\rightarrow Y$ satisfy condition $d_Y(f(x),f(x'))\leq p_2(d_X(x,x'))$ for $p_2$ nondecreasing and unbounded~\footnote{such f is called \emph{uniformly expansive}.}. Suppose that trasdimY$<\infty$ and that family $\Gamma$ of preimages of bounded sets has countable Trasdim. Then trasdim$X<\infty$.\end{twr}

\begin{proof}
Fix any nondecreasing sequence of natural numbers $r=(r_i)_{i=1}^{\infty}$ and take minimal with respect to inclusion finite prefix $\tau_1\subset r$ such that $\tau_1\notin T^{\rightarrow}(\Gamma)$. Let $R_1=max\{\tau_1(i)~|~i\in\m{N}\}$. Supposing that we have already defined number $R_{j-1}$ and finite sequence $\tau_{j-1}$, the sequence $\tau_{j}$ and the number $R_j$ are defined in similar way as in case $j=1$ by considering sequence $r\backslash\bigcup_{k=1}^{j-1}\tau_k$ instead of r - it's a little abuse in notation, because formally we should translate the domains, which could result in loosing the idea from sight - it's all about dividing sequence r into subsequent finite parts (finite sequences) $\tau_j$, each not belonging to $T^{\rightarrow}(\Gamma)$, and taking maximum $R_j$ of each one of them. For sequence $(R_j)_{j=1}^{\infty}$ consider sequence $(S_j)_{j=1}^{\infty}$ such that $\forall_{j\in\m{N}}S_j>p_2(R_j)$ and take minimal with respect to inclusion finite prefix $\Phi\subset (S_j)_{j=1}^{\infty}$ such that $\Phi\notin T^{\rightarrow}(Y,d_Y)$. Take $\Phi$-cover of Y and then consider preimage of this cover. Preimage of $S_j$-disjoint family is $R_j$-disjoint family consisting of elements of $\Gamma$ and so for every element of this $R_j$-disjoint family we can find its $\tau_j$-cover; $R_j$-disjointness assures that we can cover these elements simultaneously and so every such $R_j$-disjoint family has (uniformly bounded) $\tau_j$-cover. Now it's clear that there exists finite prefix of r which doesn't belong to $T^\rightarrow(X,d_X)$.
\end{proof}

Notice, that the proof implies that we can weaken conditions of our theorem. 

\begin{twr}\emph{Fibering Theorem.} Let $f:X\rightarrow Y$ be uniformly expansive function. Let $\Gamma$ be the family of preimages of bounded sets and $\Gamma_D$ a family of preimages of D-bounded sets for $D<\infty$. Suppose trasdimY$<\infty$ and the set-theoretical sum of trees $T^\rightarrow(\Gamma_D)$, literally $L:=\bigcup_{D\in\m{N}}T^\rightarrow(\Gamma_D)$, is a tree with countable rank. Then\\ trasdim$X<\infty$.\end{twr}

\begin{proof}
It suffices to repeat the proof of preceding theorem, considering tree $L$ instead of $T^{\rightarrow}(\Gamma)$ and applying the fact that the $\Phi$-cover of $Y$ is $D$-bounded for some $D<\infty$.
\end{proof}

\begin{twr}\label{twr:produkt}\emph{Product theorem.} If X,Y - metric spaces such that trasdim$X<\infty$ and trasdim$Y<\infty$ then trasdim($X\times Y)<\infty$.\end{twr}

\begin{proof}
Notice that projection is uniformly expansive function which satisfies conditions of previous theorem.
\end{proof}

\section{Examples}\label{sec:przyklady}

Firstly let's give some usefull definitions.

\begin{df}For a metric space X and sequence $(r_i)_{i=1}^n$, $(r_i)_{i=1}^n$-dimension of space X we call the least number minus 1, which player A can respond with in n-th round of dimension game $G_X$, assuming that sequence $(r_i)_{i=1}^n$ is the strategy of player B. In case of sequence of lenght one we use obvious abbreviation.\end{df}

Note that \emph{a priori} a space can have finite $(r_i)_{i=1}^n$-dimension for every finite sequence and still have uncountable trasdim.

Metric space X is called \emph{r-space} if $\forall_{x\neq x'\in X}~d_X(x,x')\geq r$. 
    
\begin{twr}\label{twr:rX}Each metric space X is for every $r>0$ coarsly equivalent with some r-space.\end{twr}

\begin{proof}
Indeed, notice that from Zorn's Lemma we have maximal in repsect to inclusion r-disjoint 0-bounded\footnote{It consists of single points.} family $\mc{U}$. From maximality of $\mc{U}$ we have that $X\subset D_r(\bigcup\mc{U})$ and so X$=D_r(\bigcup\mc{U})$. Straightforwardly from the definition of coarse equivalence each space is coarsly equivalent with its r-hull (generalized ball of radius r) for any $r>0$.
\end{proof}

By $k\m{Z}$ we denote space $\{kl~|~l\in\m{Z}\}$ with standard metric. Generally, by $kX$ we shall understand k-space coarsly equivalent with space $X$.
\\\\
THREE FUNDAMENTAL FACTS (TFF):

\begin{enumerate}
	\item trasdim$\m{Z}^n=n$ (see~\cite{Rad})
	\item $k\m{Z}^n$ has 2k-dimension equal to n. (see~\cite{Rad})
	\item For any $s\geq 2k$ there is no cover of cube $[-s,s]^n\cap k\m{Z}^n$ constisting of n families 2k-disjoint and s-bounded. (see~\cite{Rad}, prove of Lemma~3)\footnote{Note that this prove bases on contradiction with Lebesgue Covering Theorem and therefore it's a fact of classical dimension theory.}
\end{enumerate}

Naturally $\forall_{k\in\m{N}}$~trasdim$k\m{Z}^n=n$ (coarse equivalence).
\\\\
Now we give two standard examples of metric spaces with trasdim$=\omega$ and trasdim$=\infty$.

\begin{itemize}
	\item trasdim$\varoplus_{i<\omega}\m{Z}_i=\infty$, where $\m{Z}_i=\m{Z}$ for every $i\in\m{N}$; assume taxicab metric. This fact follows easily from TFF.
	\item Space $\bigcup_{k=1}^{\infty}k\m{Z}^k$ treated as subspace of the former one, has trasdim$\bigcup_{k=1}^{\infty}k\m{Z}^k=\omega$. Indeed, for every $k\in\m{N}$ we can take one k-disjoint family which covers the complement of $(k-1)$-dimensional subspace.
\end{itemize} 

Now let's give more general examples.

\begin{stw}Fix sequence of metric spaces $(X_i)_{i=1}^\infty$ such that $sup\{trasdimX_i~|~i\in\m{N}\}=\alpha$. Then trasdim($as~\bigcup_{i=1}^\infty~iX_i)=\alpha$.\end{stw}

Let $B^k_r$ denote the ball in $\m{Z}^k$ centred in zero and of radius r. We give three interesting examples which undermine some naive intuitions. Proofs of following equalities are easy exercises for applying TFF.

\begin{itemize}
	\item trasdim(as~$\bigcup_{k=1}^\infty B^k_k)=\infty$
	\item trasdim(as~$\bigcup_{k=1}^\infty kB^k_k)=\omega$
	\item trasdim(as~$\bigcup_{k=1}^\infty B^n_k)=n$
\end{itemize}

There is another example. Fix any strictly increasing sequence of natural numbers $c=(c_i)_{i=1}^\infty$ and direct sum $\varoplus_{i<\omega}\m{Z}_i$ (with taxicab metric), where $\m{Z}_i=\m{Z}$ for every $i\in\m{N}$. Here we identify $\m{Z}_i$ with due subspace of $\varoplus_{i<\omega}\m{Z}_i$.
\begin{equation}\label{eq:koronny}
\bigcup~c:=\bigcup_{n=1}^\infty (c_1\m{Z}_1\varoplus c_2\m{Z}_2\varoplus\dots\varoplus c_n\m{Z}_n)
\end{equation}

Supposedly this space has trasdim $\omega$; we shall refer to the matter in~\ref{chap:hipotezy}. It would be then an example of space with trasdim $\omega$ considerably different from the former ones, which where characterised by the fact that for every $r<\infty$ there exists one r-disjoint family whose complement is finite dimensional.

\section{Sums of proper spaces}\label{sec:sumy}

\begin{df}Metric space is \emph{proper} if its every bounded and closed subspace is compact.\end{df}

Clearly, if X is discrete and proper, then it's also locally finite. Particulary (see~Note~\ref{uw:subcl}) every proper metric space is coarsly equivalent with some locally finite k-space, for any fixed $k\in\m{N}$. Note that every locally finite space is countable (unless metric has infinite values). Theory of asymptotic dimension mainly concerns with proper spaces and therefore the loss of generality in this chapter is fairly small.

\begin{uw}\label{uw:subcl}Every subset of closed and discrete subspace of a given space is also closed and discrete.\end{uw}

\begin{twr}\label{twr:zwr}\emph{Fission Theorem.} Let X - proper metric space and $\mc{B}$ - ascending family of bounded subspaces of X such that $\bigcup\mc{B}=$X. Then trasdimX$=$Trasdim$(\mc{B})$.\end{twr}

\begin{wn}For every countable ordinal $\alpha$ such that there is proper metric space X and trasdimX$=\alpha$, there exists asymptotic sum S of finite dimensional spaces (or even finite spaces) such that trasdimS$=\alpha$.\end{wn}

To prove our theorem we need Compactness Rule.

\subsection*{Compactness Rule}

I assume that the reader is familiar with definition of filter (on a set). A filter is an ultrafilter iff it's maximal filter with respect to inclusion; every filter extends to an ultrafilter (Zorn's Lemma). We recall basic facts concerning ultrafilters.

\begin{fkt}Let $\mc{U}$ - ultrafilter. If $A\notin\mc{U}$ then its complement $A'\in\mc{U}$. Note that (without prove) this property can be assumed as definition of an ultrafiler.\end{fkt}

\begin{proof}
If $A\notin\mc{U}$ then from maximality of $\mc{U}$ there exists a set $V\in\mc{U}$ such that $V\cap A=\emptyset$. But then $V\subset A'$ and so $A'\in\mc{U}$ form the definition of a filter.
\end{proof}

\begin{fkt}Every ultrafilter $\mc{U}$ is prime filter, that is, if $A\cup B\in\mc{U}$ then $A\in\mc{U}$ or $B\in\mc{U}$ (it's also condition equivalent to the definition of ultrafilter.)\end{fkt}

\begin{proof}
Suppose that $A\notin\mc{U}$. The former Fact implies that $A'\in\mc{U}$. From the definition of filter it follows that $A'\cap (A\cup B)\in\mc{U}$. But $A'\cap (A\cup B)\subset B$ and so $B\in\mc{U}$.
\end{proof}

\begin{fkt}If a family $\mc{X}$ of subsets of a set X has finite intersection property (that is, every finite intersection of sets from this family is nonempty) then this family extends to an ultrafilter on X.\end{fkt}

\begin{proof}
The reader can easily check that the family of all supersets of all finite intersections of sets from $\mc{X}$ is a filter. 
\end{proof}

Let $(A_i)_{i\in I}$ be a family of nonempty finite sets indexed by arbitrary set I. We say that a set $S\subset\bigcup_{J\subset I}(\Pi_{i\in J}A_i)$ \emph{covers finite subsets} if for every finite subset $K\subset I$ exists $g\in S$ such that $K\subset dom(g)$
\\\\
Function $f\in\Pi_{i\in I}$ is \emph{filtering extension of S} if for every finite $K\subset I$ exists $g\in S$ such that $f_{|K}=g_{|K}$.

\begin{twr}\emph{Compactness Rule.} With the notation above, if S covers finite subsets then there exists filtering extension of S.\end{twr} 

\begin{proof}
Define the family of sets $\mc{B}=\{B_i~|~i\in I\}$, where $B_i=\{g\in S~|~i\in dom(g)\}$. Since S covers finite subsets it follows that $\mc{B}$ has finite intersection property and so extends to an ultrafilter $\mc{U}$ on S. For every $i\in I$ and $a\in A_i$ let $B_i^a=\{g\in B_i~|~g(i)=a\}$. Clearly $B_i=\bigcup_{a\in A_i}B_i^a$ and for every $i\in I$ it's a finite sum. Ultrafilters are prime and so there exists $a_i\in A_i$ such that $B_i^{a_i}\in\mc{U}$. Define $f\in\Pi_{i\in I}A_i$ in the following way: $f(i)=a_i$. From properties of filters we conclude that f is indeed a filtering extension of S.
\end{proof}

\begin{df}The \emph{trace} on a subset $Y\subset X$ of a family $\mc{U}$ of subsets of a set X, we call the family of traces on Y of sets belonging to $\mc{U}$. The trace of a $\tau$-covering $\tau$ of set X is - similary - the sequence of traces of subsequent families belonging to $\tau$ (we preserve the ordering of the families).\end{df}

\begin{proof}
(of Theorem~\ref{twr:zwr}) Firstly let's assume that X is locally finite.

If for finite sequence $\tau$ there exists $\tau$-covering of space X then due families of traces of this covering create uniformly bounded family of uniformly bounded covers of elements of $\mc{B}$. By contraposition we therefore have $T^\rightarrow(\mc{B})\subset T^\rightarrow(X,d)$ and so Trasdim$\mc{B}\leq$trasdimX.

Now suppose that for some finite sequence $\tau$ and some real number $D<\infty$ there exists D-bounded family of $\tau$-coverings of sets belonging to $\mc{B}$. Since every set $b\in\mc{B}$ is finite, the set of all D-bounded $\tau$-coverings of set b is finite, and from our assumptions - also nonempty. Denote the latter set by $A_b$. Let $S\subset\bigcup_{B\subset\mc{B}}(\Pi_{b\in B}A_b)$ be the set of all functions $g\in\bigcup_{B\subset\mc{B}}(\Pi_{b\in B}A_b)$ such that for every pair $c,d\in dom(g)$ such that $c\subset d$ the trace on b of $\tau$-covering $g(d)$ of set d is equal to $g(b)$, briefly: $g(d)_{|b}=g(b)$. 

Since $\mc{B}$ is ascending family of sets, our assumptions imply that S covers finite sums. From the Compactness Rule there exists filtering extension of S, denoted by $f\in\Pi_{b\in\mc{B}}A_b$. We must show that f determines D-bounded $\tau$-covering of space X. It suffices to show that for $k\in dom(\tau)$ the family of $\tau(k)$-families\footnote{Recall that $\tau$-covering is a sequence of families and so $\tau(k)$-family is a k-th family in respective $\tau$-covering. Let's stick to this notation for clarity.} given by $f$ determines D-bounded $\tau(k)$-family of subsets of X.

Choose any $b\in\mc{B}$ and consider $\tau(k)$-family belonging to $\tau$-covering $f(b)$ of b. Take any set P belonging to this family and consider the family of sets $P_d$ (it's defined uniquely) indexed by sets d such that $b\subset d$ and assure that $P_d$ belongs to $\tau(k)$-family of $\tau$-covering $f(d)$ of d and $P_d\cap b=P$. Define set $X_P:=\bigcup_{d\in\mc{B}; d\supset b}P_d$. Due $\tau(k)$-family for the whole spcace X is equal to $\bigcup_P \{X_P\}$. It's easy to check that it is D-bounded and $\tau(k)$-disjoint family and the finite sum of such families indexed by $k\in dom(\tau)$ is indeed a $\tau$-covering of. Therefore we have $T^\rightarrow(X,d)\subset T^\rightarrow(\mc{B})$ and Trasdim$\mc{B}\geq$trasdimX follows.

Now let X be proper, but not neccesarily locally finite. Fix any discrete and closed subspace of X which is coarsly equivalent to X (see: Theorem~\ref{twr:rX}) and denote it by DX. Put $D\mc{B}:=\{b\cap DX~|~b\in\mc{B}\}$. We have the following sequence of inequalities:
$trasdim(DX)=Trasdim(D\mc{B})\leq Trasdim(\mc{B})\leq trasdim(X)=trasdim(DX)$.
\end{proof}

The assumptions of Theorem~\ref{twr:zwr} can be refined.

\begin{twr}Let X be a proper metric space and $\mc{B}$ - family of bounded subspaces of X such that $\bigcup\mc{B}=$X and $\mc{B}$ has supersets of any finite sum of its elements (we say that $\mc{B}$ covers finite sums). Then trasdimX$=$Trasdim$(\mc{B})$.\end{twr}

\begin{proof}
Repeat previous proof to the moment pointed below. We assume that X is locally finite. Let $S\subset\bigcup_{B\subset\mc{B}}(\Pi_{b\in B}A_b)$ be a set of functions $g\in\bigcup_{B\subset\mc{B}}(\Pi_{b\in B}A_b)$ such that for any pair $c,d\in dom(g)$, traces of $\tau$-coverings $g(d)$ and $g(b)$ of d and b respectively are identical on the intersection $b\cap d$, briefly $g(d)_{|b\cap d}=g(b)_{|b\cap d}$ (we say that such functions define \emph{consistent} $\tau$-coverings/families). $\mc{B}$ covers finite sums which implies that S covers finite subsets and so there exists filtering extension f. We must define due $\tau(k)$-covering of the whole space X. It suffices to define due $\tau(k)$-family for every $k\in dom(\tau)$. 

Function f defines consistent $\tau(k)$-families on elements of $\mc{B}$. Choose any set $b\in\mc{B}$ and any element P of this set. Take from $\mc{B}$ all maximal chains (with respect to inclusion) of supersets of b. Similary as in the former proof, for every such maximal chain c we can define D-bounded set $P_c$, because ascending sum of D-bounded sets is D-bounded itself. 

Set $X_P:=\bigcup_c~P_c$ is D-bouded, too. Indeed, suppose contrary. Then there exist points $x,x'\in X_P$ such that $d(x,x')>D$. This means, there exist sets $P_x\supset b$ and $P_{x'}\supset b$ such that x belongs to (unique) element A belonging to $\tau(k)$-family of the cover $f(P_x)$ and x' belongs to (unique) element A' belonging to $\tau(k)$-family of cover $f(P_{x'})$. From covering of finite sums it follows however, that there exists maximal chain c', which contains the set $d\supset P_x\cup P_{x'}$. From the consistency of $\tau(k)$-families defined by f we have, that there exists unique element of $\tau(k)$-family of cover $f(d)$, which is the superset of the set P. It must be therefore superset of set $A\cup A'$ and it's not D-bounded, contradiction.

It remains to show that sets $X_P$ indexed by sets P create $\tau(k)$-disjoint family. Suppose that there exists $p\in X_P$ i $q\in X_Q$ such that $d(p,q)<\tau(k)$ and P belongs to $\tau(k)$-family of cover $f(b_P)$ for some set $b_P\in\mc{B}$, similary Q belongs to $\tau(k)$-family of cover $f(b_Q)$. There exist supersets $d_P\supset b_P$ and $d_Q\supset b_Q$, which contain - as elements of respective $\tau(k)$-families - unique sets $A_P\supset P$ and $A_Q\supset Q$ such that $p\in A_P$ and $q\in A_Q$. Consider the superset $D\supset d_P\cup d_Q$ such that $D\in\mc{B}$. Since $d(p,q)<\tau(k)$, some superset $E\supset A_P\cup A_Q$ belongs $\tau(k)$-family of cover $f(D)$. 

It's easy to see (from covering of finite sums) that for any sets (for which the following has sense), if $U\subset V$ then $X_U=X_V$. Therefore $X_P=X_E=X_Q$. 

The fact that such families for every $k\in dom(\tau)$ create together a $\tau$-covering of the space X is easily implied by the second condition of our theorem. There ramaining part of the proof goes just as before.
\end{proof}

Similary we prove the following, more general

\begin{twr}Let $\Gamma$ - the family of proper, bounded metric spaces, covering finite sums and such that Trasdim$\Gamma=\alpha$. Then trasdim$(\bigcup\Gamma)=\alpha$(it can also be
$\alpha=\infty$).\end{twr}

\begin{proof}
We set X:=$\bigcup\Gamma$. Subspace DX is not necessarily locally finite, but $b\cap DX$ for $b\in\Gamma$ is finite set anyway. The rest of the proof goes as before.
\end{proof}

Now we can prove

\begin{twr}Let $\Gamma$ - family of proper spaces covering finite sums and let Trasdim$\Gamma=\alpha$. Then trasdim$(\bigcup\Gamma)=\alpha$(it can also be
$\alpha=\infty$)\end{twr}

\begin{proof}
It suffices to show that $\bigcup\Gamma$ is equal to sum of some family $\Theta$ of proper and bounded spaces, such that $\Theta$ covers finite sums and Trasdim$\Theta=\alpha$. In the space $\bigcup\Gamma$ fix concentric sequence of balls $(B_i)_{i=1}^\infty$ of radious raising to infinity. Define $\Theta=\{\gamma\cap B_i~|~\gamma\in\Gamma, i\in\m{N}\}$.
\end{proof}

Note that these proofs imply some interesting facts:
\\\\
We say that the space X is \emph{coarsly separable} if it contains countable subspace DX such that $X=D_R(DX)$ for some $R<\infty$.

\begin{stw}Ascending family of proper spaces is coarsly separable. Moreover, every discrete and closed subset of such sum is countable.\end{stw}

\begin{proof}
Take any discrete and closed subset of this sum. The preceding proofs imply that this set is countable (indexed by $i\in\m{N}$) sum of ascending sums of finite sets. Such sum is always countable.
\end{proof}

\chapter{Open problems}

\section{Closure under products}\label{sec:produkt}

Problem stated~in~\cite{Pearl}, for solution see theorem~\ref{twr:produkt}.

\section{Connection with FDC}

The reader is ecouraged to familiarize with~\cite{GTY}. There is neat connection between asymptotic properties C and FDC. It was well know that asymptotic property C implies asymptotic property FDC, however the opposite implication hadn't been determined up to not long ago (it was suggested by Dranishnikow himself). In Crocow took place the lecture (see.~\cite{Zar}) in which Zarychnyj proved the strict inclusion $C\subset FDC$. As far as I know Zarychnyj did not publish proper article yet.

\section{Find the space of a given countable dimension trasdim}\label{sec:konstrukcja}

Open problem stated in~\cite{Pearl}. I give partial answer and buch of intuitions in chapter~\ref{sec:przyklady} (Examples) and in chapter~\ref{chap:hipotezy}. Also chapter~\ref{sec:sumy} may be useful here, since it turns out that in case of proper spaces one can restrict himself, without loss of generality, to asymptotic sums (moreover one can sum finite spaces). It's also important here to solve problem~\ref{sec:monotone}.

\section{Monotonicity}\label{sec:monotone}

Is it true that for any countable ordinal $\alpha$ (or for $\alpha=\infty$) and for any countable $\beta<\alpha$, in every space X with dimension trasdimX$=\alpha$ there exists subspace Y with dimension trasdimY$=\beta$?

\section{Optimal bound for dimension of finite sum of spaces}\label{sec:ogrsumy}

As in the title. Such bound is interesting itself and would be useful in many aspects. For example in assertion~\ref{stw:assum} one could - perhaps - get rid of the assumption that the spaces are bounded.

\chapter{Hipothesis and some comments}\label{chap:hipotezy}

Let's go back to the example~\ref{eq:koronny} from chapter~\ref{sec:przyklady}.

\begin{hip}\label{hip:bigcupc}For every increasing sequence of natural numbers c, trasdim(~$\bigcup~c)=\omega$.\end{hip}

Argumentation. Fix c and consider dimension game for the space $\bigcup~c$. It suffices to show that in any dimension game the second round is a moment of stabilization. Let's use the notation from chapter~\ref{sec:gra}. Fix $r_1$ and suppose that $r_1>2c_i$. Considering the subspace $c_1\m{Z}_1\varoplus c_2\m{Z}_2\varoplus\dots\varoplus c_{i}\m{Z}_{i}$ and keeping in mind TFF (see chapter~\ref{sec:przyklady}) it must be $k_1>i$. We should show that $k_2$ can be determined independently of $r_2$. Find the least j such that $c_j>r_1$. Fix $r_2$ and find the least l such that $c_l>r_2$. Without loss of generality we assume that $l>j$. It suffices to show that every subspace of type $X:=c_1\m{Z}_1\varoplus c_2\m{Z}_2\varoplus\dots\varoplus c_{l}\m{Z}_{l}$ can be covered with adequately disjoint family of power $j+2$. Here the intuition is such that for every fixed $r_2$ all subspaces of X of type $c_1\m{Z}_1\varoplus c_2\m{Z}_2\varoplus\dots\varoplus c_{j}\m{Z}_{j}$ (briefly we say about j-type subspace) can be simultaniously ''strongly'' holed with one $r_1$-disjoint family consisting of adequately large j-dimensional cubes. How ''strongly''? The idea is as follows - we pick one such j-type subspace, hole it with large cubes and then cover the complement of such holed j-type subspace with $r_2$-cover of power $(j+1)$. The j-type subspaces adjacent to ours (keep in mind we are in X) are covered identically (with respect to the obvious isomorphism) with the only difference that all families are translated by some specific vector. Our goal is to achieve such ''holing'' of the space X (with the $r_1$-family) that the sum of corresponding $r_2$-families of all j-type subspaces is $r_2$-disjoint family. Thus we would be able to cover the hole X with adequately disjoint cover of power $(j+2)$. End of argumentation.
\\\\
One interesting problem is the problem~\ref{sec:konstrukcja}. The sole possibility of positive answer to the problem~\ref{sec:monotone} raises the importance of finding the space of trasdim $\omega+1$. The reader can easily see that the naive try to repeat the trick with asymptotic sums for limit $\alpha$ in case of successor $\alpha$ soon raise difficulties - the constructed spaces turn out to have trasdim$=\infty$. Such tries and the example~\ref{eq:koronny} suggest the hipothesis:

\begin{hip}\emph{Omega conjecture}. If $\omega\leq trasdimX<\infty$ then $trasdimX=\omega$.\end{hip}

Various intuitions behind the omega conjecture can be summed up in one saying that the necessity of fixing two numbers implies that the space is already to ''dense asymptotically'' and has the dimension $\infty$ as in the naive tries of creating the space with dimension $\omega+1$. For example the intuition presented in the argumentation behind the hipothesis~\ref{hip:bigcupc} says that if player A has the winning strategy, then he is always able to win ''with one family''. So one r-disjoint family seems to be very strong tool. The hipothesis however is in the first place the interesting course of study and it inspired most of my work.

\end{document}